\documentclass[12pt]{amsart}
\usepackage{amsfonts}
\usepackage{enumerate}
\usepackage{helvet,courier}
\usepackage{mathptmx,amsmath}
\usepackage{type1cm}
\DeclareMathAlphabet{\mathcal}{OMS}{cmsy}{m}{n}
\DeclareMathAlphabet{\mathbbold}{U}{bbold}{m}{n}  
\usepackage{amsthm}
\usepackage{graphicx}       
\usepackage{multicol}
\usepackage{mathrsfs}
\usepackage{amssymb}
\usepackage{verbatim}
\usepackage{esint}

\usepackage[usenames,dvipsnames]{xcolor}

\theoremstyle{plain}
\newtheorem{thm}{Theorem}[section]
\newtheorem{lm}[thm]{Lemma}
\newtheorem{cor}[thm]{Corollary}

\theoremstyle{remark}
\newtheorem{rmk}{Remark}

\theoremstyle{definition}

\newcommand{\bnu}{\begin{enumerate}}
\newcommand{\enu}{\end{enumerate}}
\newcommand{\bpf}{\begin{proof}}
\newcommand{\epf}{\end{proof}}

\newcommand{\sset}{\subset}

\newcommand{\om}{\omega}
\newcommand{\Om}{\Omega}

\newcommand{\ep}{\epsilon}

\newcommand{\tht}{\theta}

\newcommand{\vp}{\varphi}
\newcommand{\de}{\delta}

\newcommand{\bbz}{\mathbb{Z}}

\newcommand{\bbr}{\mathbb{R}}

\newcommand{\f}{\frac}

\newcommand{\nf}{\infty}

\newcommand{\tf}{\tfrac}
\newcommand{\wh}{\widehat}

\newcounter{question}
\newcommand{\qt}{%
        \stepcounter{question}%
        \thequestion}
\newcommand{\bq}{\fbox{Q\qt}\ }

\newcommand{\les}{\lesssim}

\allowdisplaybreaks

\makeindex         

\begin{document}

\author[P. Chen]{Peng Chen}
\address{Department of Mathematics, Sun Yat-sen  University, Guangzhou, 510275, P. R. China}
\email{chenpeng3@mail.sysu.edu.cn}

\author[D. He]{Danqing He}
\address{Department of Mathematics, Sun Yat-sen  University, Guangzhou, 510275, P. R. China}
\email{hedanqing@mail.sysu.edu.cn}

\author[L. Song]{Liang Song}
\address{Department of Mathematics, Sun Yat-sen  University, Guangzhou, 510275, P. R. China}
\email{songl@mail.sysu.edu.cn}

\title{Weighted inequalities of bilinear rough singular integrals}
\date{}

\thanks{
The first author was supported by NNSF of China (No. 11501583), Guangdong Natural Science Foundation
(No. 2016A030313351) and the Fundamental Research Funds for the Central Universities (No. 161gpy45). 
The second author was supported by the
Fundamental Research Funds for the Central Universities (No. 20173400031610259).
The third author was supported by NNSF of China (No. 11471338  and 11622113).}
 
\maketitle

\begin{abstract}
We establish a quantitative weighted inequality for the
bilinear rough singular integral, where the 
bound is controlled by the cube of the  weight constant.

\end{abstract}




\section{Introduction}
The optimal norms  of linear and multilinear weighted
inequalities of kinds of operators attracted a lot of attention in  past decades.
The $A_2$ conjecture asks if the weighted norm of a (smooth) Calder\'on-Zygmund operator
depends on the weight constant $[\om]_{A_2}$ linearly. This was solved by Hyt\"onen \cite{Hytonen2012}.
Lerner \cite{Lerner2013} attacked this problem later using  the sparse operators.
That is why this method is often referred as the sparse method/control.
A even shorter proof of the sparse control for singular integral operators with kernels
satisfying the Dini condition was given by Lacey \cite{Lacey2017}, which inspired a lot of other papers, 
\cite{Culiuc2016}, \cite{Krause2016}, \cite{Conde2016}, and \cite{Lacey2016}, just to name  a few.

There is a natural question after the solution to the $A_2$ conjecture. Does the weighted norm of a rough
singular integral depend on the weight constant linearly as well? This question is partially answered by Hyt\"onen, Roncal, and Tapiola \cite{Hytonen2015}, who proved that the weighted norm of a rough
singular integral is bounded by 
$C[\om]_{A_2}^2$.
Their method depends on a modification of a classical dyadic decomposition, see \cite{DRdF} for instance, and Lacey \cite{Lacey2017}.

In this note we generalize the result of \cite{Hytonen2015}  to the bilinear setting.

A bilinear rough singular integral is defined by
\begin{equation}
T_\Om(f,g)(x)
=\int_{\bbr^{2n}}\f{\Om((y,z)')}{|(y,z)|^{2n}}f(x-y)g(x-z)dydz,
\end{equation}
where $(y,z)'=\tf{(y,z)}{|(y,z)|}\in\mathbb S^{2n-1}$, the unit sphere in $\mathbb R^{2n}$,
and $\Om$ is an $L^\nf$ function defined on $\mathbb S^{2n-1}$ with vanishing integral, namely
$\int_{\mathbb S^{2n-1}}\Om=0$.
The boundedness of this operator goes back to Coifman and Meyer
\cite{Coifman1975}, which was  proved for all points except for the endpoints by Grafakos, He, and Honz\'ik  \cite{Grafakos2015}.

We are interested in the weighted norm inequality for $T_\Om$, namely
\begin{equation}\label{e05231}
\|T_\Om(f,g)\|_{L^1(\nu)}
\le C_{[\vec\om]_{A_{(2,2)}}}
\|f\|_{L^2(\om_1)}
\|g\|_{L^2(\om_2)},
\end{equation}
where $(\om_1,\om_2)$ is an $A_{(2,2)}$ weight; see \eqref{e07058} below
for the definition.
We are concerning how the constant $C_{[\vec\om]_{A_{(2,2)}}}$
depends on $[\vec\om]_{A_{(2,2)}}$.

The
weighted norm inequality for bilinear rough singular integrals has been addressed by
\cite{Cruz-Uribe2016}
and
\cite{Barron2017}.
Cruz-Uribe and Naibo \cite{Cruz-Uribe2016} obtained the first weighted inequality of bilinear rough singular integrals
via interpolation between measures.
Barron \cite{Barron2017} 
obtained a sparse control of $T_\Om$, which implies that
the weighted norm of a bilinear rough singular integral depends on the weight constant of the
bilinear weight. However, no explicit expression was provided in both papers for the 
classical multiple weights introduced in \cite{Lerner2009}. 

We choose a different way here to give
an explicit expression showing how the weighted norm $\|T_\Om\|_{L^{2}(\om_1)\times L^{2}(\om_2)\to
L^1(\nu)}$ depends on the corresponding weight constant.
Our method could be modified to other points $(p_1,p_2,p)$ beyond $(2,2,1)$ we study
in this note, but we will not pursue them here since even for the $(2,2,1)$ 
case we cannot obtain the best result,
which we conjecture as $[\om_1,\om_2]_{A_{(2,2)}}$.
The reader will find that our method relies heavily on the idea
Hyt{\"o}nen, P{\'e}rez, and Rela
 \cite{Hytonen2015}
used to handle the linear version.

Our main result is the following theorem.

\begin{thm}\label{07052}

Let $T_\Om$ be a bilinear rough singular integral operator with $\Om\in L^\nf(\mathbb S^{2n-1})$ and $\int_{\mathbb S^{2n-1}}\Om=0$,
then 
\begin{equation}\label{e07056}
\|T_\Om(f,g)\|_{L^1(\nu)}
\le C[\om_1,\om_2]^3_{A_{(2,2)}}\|f\|_{L^2(\om_1)}
\|g\|_{L^2(\om_2)}
\end{equation}
whenever $(\om_1,\om_2)\in A_{(2,2)}$ and $\nu=\om_1^{1/2}\om_2^{1/2}$.

\end{thm}

\section{ Reverse H\"older inequalities}

Let us recall some basic definitions. A local integrable nonnegative function
$\om$ is an $A_p$ weight for $1< p<\nf$ if 
$$
[\om]_{A_p}=\sup_Q\fint_Q\om(\fint_Q\om^{1-p'})^{p-1}<\nf.
$$
The constant $[\om]_{A_p}$ is referred as the weight constant of $\om$.
For the case $p=\nf$, we take the definition
$$[\om]_{B_\nf}=\sup_Q\f1{\om(Q)}\int_QM(\om\chi_Q)<\nf;$$
see \cite{Hytonen2012a} for instance. 
The class $B_\nf$ coincides with the classical weight class $A_\nf=\cup_{1\le p<\nf}A_p$ 
since $1\le [\om]_{B_\nf}\le C[\om]_{A_p}$ (see \cite{Hytonen2013a})
and $\om\in B_\nf$ satisfies the reverse H\"older inequality (Lemma~\ref{11301}).
For some technical reason, we introduce also
$$(\om)_{A_p}
=\max ([\om]_{B_\nf},[\om^{1-p'}]_{B_\nf})
\le C [\om]_{A_p}^{\max(1,\tf{p'}p)}.$$

A remarkable property of weights is that they satisfy the reverse H\"older inequality (RHI), which states that there exists a positive
$\ep$ such that 
$$\fint_{Q_0}\om^{1+\ep}dx\le C\big(\fint_{Q_0}\om dx\big)^{1+\ep}.
$$
As a simple corollary we see that $\om^{1+\ep}\in A_p$ when $\om\in A_p$.

A range of $\ep$ is given by the following lemma proved in \cite{Hytonen2012a}.

\begin{lm}[{\cite[Theorem 2.3]{Hytonen2012a}}]\label{11301}
Let $\om\in B_\nf$ and let $Q_0$ be a cube. Then 
\begin{equation}\label{e070510}
\fint_{Q_0}\om^{1+\ep}dx\le 2\big(\fint_{Q_0}\om dx\big)^{1+\ep}
\end{equation}
for any $\ep>0$ such that $0<\ep\le\f1{2^{n+1}[\om]_{B_\nf}-1}$.

In particular, for $p\ge 2$, \eqref{e070510} holds for $\ep\le C [\om]_{A_p}^{-1}$.



\end{lm}

\begin{lm}[{\cite[Corollary 3.16]{Hytonen2015}}]\label{11181}

Let $\om\in A_p$ with $p\in (1,\nf)$, then there exists a constant $\de\sim 1/(\om)_{A_p}$
such that 
\begin{equation}\label{e07059}
[\om^{1+\de}]_{A_p}\le 4[\om]_{A_p}^{1+\de}.
\end{equation}

\end{lm}

This result follows from the reverse H\"older inequality. We 
 refer abusively \eqref{e07059} as a reverse H\"older inequality as well. 
A main reason for doing this as we will see below is that 
this is a good substitute for  reverse H\"older inequality of 
multiple weights,
while the generalization of \eqref{e070510} to the multiple weights
is unclear.

A multiple weight  is defined as follows.
Let
$1\le p_1,\dots, p_m<\nf$, $\tf 1p=\sum_{j=1}^m\tf 1{p_j}$,
$\vec P=(p_1,\dots, p_m)$, and $\vec\om=(\om_1,\dots,\om_m)$. Set $\nu=\nu_{\vec\om}
=\prod_{j=1}^m\om_j^{p/p_j}$. We say $\vec\om$ satisfies $A_{\vec P}$ condition
if
\begin{equation}\label{e07058}
[\vec \om]_{A_{\vec P}}^{1/p}=\sup_Q\big(\f1{|Q|}\int_Q\nu\big)^{1/p}\prod_{j=1}^m\big(\f1{|Q|}\int_Q \om_j^{1-p_j'}\big)^{1/p_j'}<\nf.
\end{equation}
This coincides with the classical weights when $m=1$
and the supremum is $[\om]_{A_p}^{1/p}$.
The special case we are interested in is $m=2$, $p_1=p_2=2$,
where the supremum reads
$$
[(\om_1,\om_2)]_{A_{(2,2)}}=\sup_Q\big(\f1{|Q|}\int_Q\om_1^{1/2}\om_2^{1/2}\big)\prod_{j=1}^2\big(\f1{|Q|}\int_Q \om_j^{-1}\big)^{1/2}.
$$
In particular if $\om_1=\om_2$, we have
$[(\om_1,\om_2)]_{A_{(2,2)}}=[\om_1]_{A_2}$.

There is an interesting characterization of multiple weights.

\begin{lm}[{\cite[Theorem 3.6]{Lerner2009}}]\label{11182}
$\vec\om\in A_{\vec P}$ if and only if $\om_j^{1-p_j'}\in A_{mp'_j}$
and $\nu\in A_{mp}$.

\end{lm}

In particular, we have the following corollary.
\begin{cor}\label{11183}
$\vec\om\in A_{(2,2)}$ if and only if
$\om_j^{-1}\in A_{4}$ for $j=1,2$ and $\nu=\om_1^{1/2}\om_2^{1/2}\in A_2$.

\end{cor}
By Lemma \ref{11181} we know that 
for 
$$\de\les \min([\om_1^{-1}]^{-1}_{A_4},
[\om_2^{-1}]^{-1}_{A_4},[\nu]^{-1}_{A_2})\sim 
1/\max([\om_1^{-1}]_{A_4}, [\om_2^{-1}]_{A_4},[\nu]_{A_2}),$$
we have
$\om_1^{-(1+\de)}, \ \om_2^{-(1+\de)}\in A_4$
and $\nu^{1+\de}\in A_2$.
Consequently, by Corollary \ref{11183},
$\vec\om^{1+\de}=(\om_1^{1+\de}, \om_2^{1+\de})
\in A_{(2,2)}$.
This indicates the possible validity of the reverse H\"older inequality of
multiple weights
of the following form,
\begin{equation}\label{e12061}
\Big[\fint_Q\nu^{1+r}(\fint_Q\om_1^{-(1+r)})^{\tf12}(\fint_Q\om_2^{-(1+r)})^{\tf12}
\Big]^{\tf1{1+r}}\le C
\fint_Q\nu(\fint_Q\om_1^{-1})^{\tf12}(\fint_Q\om_2^{-1})^{\tf12}.
\end{equation}
In particular we have
$(\om_1^{1+r},\om_2^{1+r})\in A_{(2,2)}$
and
\begin{equation}\label{e11191}
[\om_1^{1+r},\om_2^{1+r}]_{A_{(2,2)}}\le [\om_1,\om_2]^{1+r}_{A_{(2,2)}}.
\end{equation}

\medskip

\begin{rmk}\label{12051}
By the proof of \cite[Theorem 3.6]{Lerner2009}
we see
\begin{equation}\label{e01211}
[\om_1^{-1}]_{A_4},[\om_2^{-1}]_{A_4}, [\nu]_{A_2}
\le [\om_1,\om_2]_{A_{(2,2)}}^2,
\end{equation} 
which implies that
$$
[\om_1,\om_2]_{A_{(2,2)}}^{-2}
\le \min
([\om_1^{-1}]^{-1}_{A_4}, [\om_2^{-1}]^{-1}_{A_4},[\nu]^{-1}_{A_2}).
$$
Moreover
by Lemma \ref{11301} we see that
 $[\vec\om^{1+r}]_{A_{(2,2)}}\le C [\vec\om]^{(1+r)}_{A_{(2,2)}}$ holds at least for $r\sim [\om_1,\om_2]_{A_{(2,2)}}^{-2}$.

\end{rmk}
We are concerning the largest possible number
$r$ such that $\vec\om^{1+r}\in A_{(2,2)}$.
The example $\om_1=\om_2\in A_2$ suggests that $r=[\om_1,\om_2]^{-1}_{A_{(2,2)}}$
might be a reasonable conjecture, which unfortunately turns out to be wrong.

\begin{rmk}
We observe that 
\eqref{e01211}
is sharp in the sense that
the smallest $t$ such that $[\om_1^{-1}]_{A_4},[\om_2^{-1}]_{A_4}, [\nu]_{A_2}
\le [\om_1,\om_2]_{A_{(2,2)}}^t$ is $2$,
although the number $t$ could be $1$ in the special case $\om_1=\om_2\in A_2$. 
So it is impossible to obtain a larger $r$ by improving
the exponent $t$.

To illustrate the claimed sharpness, we consider the special case 
$\om_1^{-1}=|x|^a\in A_4, \om_2=1$ and $\nu=|x|^{-\tf a2}\in A_2$.
A simple calculation (\cite[p. 506]{Grafakos2014}) shows that 
$a\in (-n,2n)$, $[\om]_{A_4}\sim \tf1{(a+n)(3n-a)^3}$, and $[\om_1,1]_{A_{(2,2)}}
=[\om]_{A_3}^{1/2}\sim \tf1{(a+n)^{1/2}(2n-a)}$.
$[\om]_{A_4}\le C[\om_1,1]^t_{A_{(2,2)}}$ is valid
only if $t\ge 2$ by letting $a\to -n$.

\end{rmk}


By the example in last remark, we are able to show that $r\sim [\om_1,\om_2]_{A_{(2,2)}}^{-2}$ we obtained in Remark~\ref{12051} is sharp.

\begin{lm}\label{08251}
If $r$ is a positive number such that \eqref{e11191} holds 
for all $(\om_1,\om_2)\in A_{(2,2)}$,
then we necessarily have 
$r\lesssim [\om_1,\om_2]_{A_{(2,2)}}^{-2}$.
In particular, \eqref{e11191} holds if and only if $r\lesssim [\om_1,\om_2]_{A_{(2,2)}}^{-2}$.

\end{lm}

\bpf
Take $\om_1^{-1}=\om=|x|^a\in A_4$, $\om_2=1$, and $\nu=\om^{-1/2}\in A_2$, or equivalently
$a\in(-n,2n)$.
We know that in this case $[\om_1,\om_2]_{A_{(2,2)}}=[\om]_{A_3}^{1/2}$, and 
\eqref{e11191} becomes
$$
[\om^{1+r}]_{A_3}\le C[\om]^{1+r}_{A_3}.
$$
A weaker version is that $\om^{1+r}\in A_3$, which is equivalent to that $(1+r)a\in(-n,2n)$. Consider the case
$a$ is close to $-n$, then
$r\le \tf{a+n}{-a}\sim [\om]_{A_3}^{-1}=[\om_1,\om_2]^{-2}_{A_{(2,2)}}$ since $[\om]_{A_3}\sim \tf1{(n+a)(2n-a)^2}.$

The last statement follows from the first one and Remark~\ref{12051}.

\epf

\begin{rmk}
There is a different property which is also called RHI for multiple weights. We refer  interested
readers to \cite{Cruz-Uribe2017}.
\end{rmk}

\section{A quantitative weighted inequality}

In this section, we prove  Theorem~\ref{07052},
which relies on an improved Dini estimate.

Let $K_0=\tf{\Om((y,z)')}{|(y,z)|^{2n}}\chi_{1\le|(y,z)|\le 2}$ be the truncated kernel of the bilinear
rough singular integral. 
Take $\vp\in\mathcal S$ such that $supp\ \wh\vp\sset B(0,1)$
and $\wh\vp(y,z)=1$ when $|(y,z)|\le 1/2$,
and define $\wh\psi=\wh\vp(\cdot)-\wh\vp(2\cdot)$
Define the kernel $K_k=2^{-2nk}K_0(2^{-k}(y,z))$, $\vp_k=2^{-2kn}\vp(2^{-k}(y,z))$, and $\psi_k=2^{-2kn}\psi(2^{-k}(y,z))$.
Define
$$
T_0(f,g)(x)=\sum_k K_k*(f\otimes g)*\vp_k(x,x),
$$
and 
$$
T_j(f,g)(x)=\sum_k K_k*(f\otimes g)*\psi_{k-j}(x,x)
$$
for $j\ge 1$.
We remark that this decomposition is essentially the same as 
the one used in \cite{Grafakos2014a},
where $K_0$ is a smooth truncation. Both truncations
satisfy the same decay condition (in frequency side),
so the argument in \cite{Grafakos2014a}
could be applied here as well.

We have 
the following lemma 
on $T_j$.

\begin{lm}[{\cite[Proposition 5]{Grafakos2015}}]\label{11191}
$T_\Om=\sum_{j\in \mathbb Z}T_j$. $T_j$ is a bilinear Calder\'on-Zygmund operator
such that
 $\|T_j\|_{L^2\times L^2\to L^1}\le C 2^{-|j|\de}$
 for $\de=1/16$.
 
Moreover,
 for any $\ep>0$, there exists a constant $C_\ep\le \tf C{\ep}$
such that $T_j$ has the Calder\'on-Zygmund constant
$C_\ep2^{|j|\ep}$ for all $j\in\mathbb Z$.

\end{lm}
\begin{rmk}
The boundedness of $T_j$ is exactly \cite[Proposition 5]{Grafakos2015}.

\cite[Lemma 11]{Grafakos2015} gives just the existence of $C_\ep$ without the
form here. To obtain the right bound we need, we have to re-examine the proof 
to show that $C_\ep\le C(\tf{1}{2n+\ep-2n}+1)\le \tf C\ep$
with the
help of \cite[Appendix B1]{Grafakos2014a}.

\end{rmk}

The Dini condition plays a crucial role in the following argument, which we
now define; see, for example, \cite{Li2016} and references therein.
A bilinear operator is called an $\om$-Calder\'on-Zygmund
operator if its kernel satisfies the size condition
$|K(x,y,z)|\le \tf {C_K}{(|x-y|+|x-z|)^{2n}}$,
and the smoothness condition
\begin{align*}
&|K(x+h,y,z)-K(x,y,z)|+|K(x,y+h,z)-K(x,y,z)|\\
+&|K(x,y,z+h)-K(x,y,z)|
\le \f1{(|x-y|+|x-z|)^{2n}}\om(\f{|h|}{|x-y|+|x-z|}),
\end{align*}
whenever $|h|\le \tf12\max(|x-y|,|x-z|)$.
We concern mainly the case when
$\om$ is increasing satisfying $\om(0)=0$,
and $\|\om\|_{Dini}=\int_0^1\om(t)\tf{dt}t<\nf$.
In this case we say the kernel (or equivalently the operator)
satisfies the Dini condition.

For $T_j$ in the previous lemma, we see that the C-Z constant
is $C_\ep2^{|j|\ep}$, which implies that we may take
$\om(t)=C_\ep 2^{|j|\ep}t^\ep$,  hence
$\|\om\|_{Dini}\le C_\ep 2^{|j|\ep}$.

\medskip

This estimate based on the classical decomposition is not good enough, and we 
need a new decomposition introduced by 
\cite{Hytonen2012a}.


Let $N(\ell)=2^\ell$, and we should define $\tilde T_\ell$
as follows.
$$
\tilde T_0(f,g)=T_0(f,g)(x)
$$
and 
$$
\tilde T_\ell(f,g)(x)=\sum_k K_k*(f\otimes g)*[\vp_{k-N(\ell)}-
\vp_{k-N(\ell-1)}](x,x)
$$
for $\ell\ge 1$.

We look at these operators one step further. We need their 
equivalent multiplier
definitions,
which are
$$
\tilde T_\ell(f,g)(x)
=\sum_{k\in\bbz}\int_{\bbr^{2n}}
m_{k,\ell}(\xi,\eta)
\wh f(\xi)\wh g(\eta)e^{2\pi ix\cdot(\xi+\eta)}d\xi d\eta,
$$
where $m_{k,\ell}(\xi,\eta)
=\wh K(2^k\xi,2^k\eta)
[\wh \vp(2^{k-N(\ell)}(\xi,\eta))-\wh \vp(2^{k-N(\ell-1)}(\xi,\eta))]$
is supported in the annulus
$$
\{(\xi,\eta)\in\bbr^{2n}:\ 2^{2^{\ell-1}}\le|(\xi,\eta)|\le 2^{2^\ell}\}.
$$

Obviously $\tilde T_\ell=\sum_{j=N(\ell-1)}^{N(\ell)}T_j$,
so we obtain the following two trivial estimates depending on
Lemma~\ref{11191}.
\begin{equation}\label{e07053}
\|\tilde T_\ell(f,g)\|_{L^1}
\le C2^{-N(\ell)\de'}\|f\|_{L^2}
\|g\|_{L^2}
\end{equation}
since 
$\sum_{j=N(\ell-1)}^{N(\ell)} C 2^{-|j|\de}
\le C2^{\ell-1}2^{-2^{\ell-1}\de}
\le C_\de 2^{-N(\ell)\de/4}$.
The second estimate is that the Calder\'on-Zygmund constant
$C_\ell$
related to
$\tilde T_\ell$ is
bounded by
$$
\sum_{j=N(\ell-1)}^{N(\ell)}  C_\ep 2^{j\ep}
\le C_\ep N(\ell)2^{N(\ell)\ep}.
$$
By taking $\ep=t\ell N(\ell)^{-1}$,
we control the last quantity by
$C_t2^{(2+t)\ell}$,
which in anyway is greater than $2^{2\ell}=N(\ell)^2$,
a bound we shall improve.



It was essentially proved in \cite{Hytonen2012a} the following lemma.

\begin{lm}[{\cite[Lemma 3.10]{Hytonen2012a}}]\label{07051}
The operator $\tilde T_\ell$ is a bilinear $\om_\ell$-Calder\'on-Zygmund
operator with
\begin{equation}\label{e07051}
C_\ell\le C\|\Om\|_{L^\nf}
\end{equation}
and 
\begin{equation}\label{e07052}
\om_\ell(t)\le C\|\Om\|_{L^\nf}\min(1, 2^{N(\ell)}t),
\end{equation}
which implies further that
$\|\om_\ell\|_{Dini}\le C\|\Om\|_{L^\nf}(1+N(\ell))$.

\end{lm}

We see that the function 
$\vp$ in \cite{Hytonen2012a} is compactly supported in the spatial side (for variable $x$), while 
our decomposition uses that $\vp$ is compactly supported in 
the frequency side (for variable $\xi$). We take $\vp$ of this form
due to the 
the method taken in \cite{Grafakos2015}.
We are still able to prove Lemma~\ref{07051}
in our setting, which is given in the Appendix.

\begin{rmk}
The Dini constant of $\tilde T_\ell$ we had was $N(\ell)^2$,
which is now $N(\ell)$.

\end{rmk}

In the linear case Lacey \cite{Lacey2017} proved sparse controls for singular integrals whose kernels satisfy the Dini condition,
which was reproved later by Hyt\"onen, Roncal, and Tapiola \cite{Hytonen2015}, and Lerner \cite{Lerner2016}. Li
\cite{Li2016} generalized Lerner's result to the multilinear setting, which is useful for us.

\begin{lm}[
{\cite[Theorem 1.2]{Li2016}}]\label{11192}
Let $T$ be a bilinear 
$\om$-Calder\'on-Zygmund operator. Then the norm
$
\|T(f,g)\|_{L^1(\nu)}$
is bounded by 
$$
C[\om_1,\om_2]_{A_{(2,2)}}(\|T\|_{L^2\times L^2\to L^1}+C_K+\|\om\|_{Dini})\|f\|_{L^2(\om_1)}
\|g\|_{L^2(\om_2)}
$$
for $(\om_1,\om_2)\in A_{(2,2)}$ and $\nu=\om_1^{1/2}\om_2^{1/2}$.

\end{lm}

Another important tool is the interpolation between measures. The version we need is taken from the classical monograph
\cite{Bergh2012}.

\begin{lm}[
{\cite[Theorem 4.4.1, Theorem 5.5.3]{Bergh2012}}]\label{11194}

Let $T$ be a bilinear operator such that 
$$
\|T(f,g)\|_{L^p(\mu_1)}
\le M_1\|f\|_{L^{p_1}(\om_1)}\|g\|_{L^{p_2}(\nu_1)}
$$
and
$$
\|T(f,g)\|_{L^p(\mu_2)}
\le M_2\|f\|_{L^{p_1}(\om_2)}\|g\|_{L^{p_2}(\nu_2)}.
$$
Then $T$ can be extended to a bilinear operator bounded
from $L^{p_1}(\om)\times L^{p_2}(\nu)$
to $L^p(\mu)$ with norm bounded by
$M_1^{1-\tht}M_2^{\tht}$,
where $\mu=\mu_1^{1-\tht}\mu_2^\tht$,
and both $\om$ and $\nu$ are defined in a similar way.

\end{lm}

Now we prove the claimed  quantitative weighted inequality of the bilinear rough singular integral. We should emphasize again that
our argument is parallel to the one previously used in 
Hyt{\"o}nen, P{\'e}rez, and Rela
 \cite{Hytonen2015} for the linear case.

\begin{proof}[Proof of Theorem~\ref{07052}]
By Lemma \ref{07051}  and 
Lemma \ref{11192}  we know that 
$$\|\tilde T_\ell(f,g)\|_{L^1(\nu)}
\le  C\|\Om\|_\nf[\om_1,\om_2]_{A_{(2,2)}}N(\ell)\|f\|_{L^2(\om_1)}\|g\|_{L^2(\om_2)}
$$
whenever $(\om_1,\om_2)\in A_{(2,2)}$. Moreover, for a fixed $(\om_1,\om_2)
\in A_{(2,2)}$, by Lemma~\ref{08251}, we have 
$(\om_1^{1+r},\om_2^{1+r})\in A_{(2,2)}$
for $r\sim [\om_1,\om_2]^{-2}_{A_{(2,2)}}$, hence
\begin{equation}\label{e11192}
\|\tilde T_\ell(f,g)\|_{L^1(\nu^{1+r})}
\le 
C \|\Om\|_{L^\nf} [\om_1^{1+r},\om_2^{1+r}]_{A_{(2,2)}}
N(\ell)
\|f\|_{L^2(\om_1^{1+r})}\|g\|_{L^2(\om_2^{1+r})}.
\end{equation}

Recall also that by \eqref{e07053} we have $\|\tilde T_\ell\|_{L^2\times L^2\to L^1}\le C\|\Om\|_\nf2^{-N(\ell)\de'}$ for a 
fixed positive
$\de'$ independent of $\|\Om\|_\nf$. 
Interpolating between this and \eqref{e11192}, using Lemma \ref{11194} 
and \eqref{e11191}, we obtain that
\begin{equation}\label{e11193}
\|\tilde T_\ell(f,g)\|_{L^1(\nu)}
\le C\|\Om\|_{L^\nf}[\om_1,\om_2]_{A_{(2,2)}}
N(\ell)^{\tf 1{1+r}}
2^{-\tf{N(\ell) \de'r}{1+r}}\|f\|_{L^2(\om_1)}\|g\|_{L^2(\om_2)}.
\end{equation}
Summing over $\ell\ge 1$, and using the argument on
\cite[p. 19]{Hytonen2012a}
we obtain that
$$\|T_\Om(f,g)\|_{L^1(\nu)}
\le  C[\om_1,\om_2]_{A_{(2,2)}} r^{-1}\|f\|_{L^2(\om_1)}\|g\|_{L^2(\om_2)},$$
which is \eqref{e07056}
and we finish the proof of Theorem~\ref{07052}.

\end{proof}

\begin{rmk}
If $r\sim [\om_1,\om_2]_{A_{(2,2)}}^{-1}$, as we conjectured right after
Remark~\ref{12051},
we obtain the weighted bound $[\om_1,\om_2]_{A_{(2,2)}}^2$,
similar to the result obtained in the linear case \cite{Hytonen2012a}.
However, $r$  can only be $[\om_1,\om_2]_{A_{(2,2)}}^{-2}$.
This indicates that  $[\om_1,\om_2]_{A_{(2,2)}}^3$
may be the limit of our method.

\end{rmk}
\section{Appendix: Proof of Lemma~\ref{07051}}
In this section we sketch the proof of Lemma~\ref{07051}.

We refer the readers to \cite[Lm 3.10]{Hytonen2012a}
for a detailed proof.
Here we just present a few tiny differences worth explaination.

A careful examination of the proof of \cite[Lm 3.10]{Hytonen2012a} shows that once we establish
(3.11) and (3.12) of \cite{Hytonen2012a}, then the remaining argument
follows smoothly.

What we want to estimate is 
$|\sum_kK_k*\vp_{k-N(\ell)}|$.
We see that
$$
K_k*\vp_{k-N(\ell)}(\vec x)=2^{-(k-N(\ell))2n}\int_{|\vec y|\sim 2^k}
\f{\Om((\vec y)')}{|\vec y|^{2n}}
\vp(\f{\vec x-\vec y}{2^{k-N(\ell)}})d\vec y,
$$
where $\vec x,\vec y\in\bbr^{2n}$.

Fix $\vec x$ and assume $|\vec x|=2^l$.
If $l\le k+10$,
then
$$
|2^{-(k-N(\ell))2n}\int_{|\vec y|\sim 2^k}
\f{\Om((\vec y)')}{|\vec y|^{2n}}
\vp(\f{\vec x-\vec y}{2^{k-N(\ell)}})d\vec y|
\le C\|\Om\|_{L^\nf}2^{-2kn}.
$$
If $l\ge k+10$,
then  $|\vp(\f{\vec x-\vec y}{2^{k-N(\ell)}})|$ is bounded
by $C \min(1, 2^{(2n+1)(k-N(\ell)-l)})$,
which implies that
\begin{align*}
&|2^{-(k-N(\ell))2n}\int_{|\vec y|\sim 2^k}
\f{\Om((\vec y)')}{|\vec y|^{2n}}
\vp(\f{\vec x-\vec y}{2^{k-N(\ell)}})d\vec y|\\
\le &C\|\Om\|_{L^\nf}2^{-(k-N(\ell))2n}
2^{-2kn}2^{(2n+1)(k-N(\ell)-l)} 2^{2kn}\\
=&C \|\Om\|_{L^\nf}2^{k-N(\ell)} 2^{-(2n+1)l}.
\end{align*}
Summing over $k$ we obtain
\begin{align*}
&|\sum_kK_k*\vp_{k-N(\ell)}|\\
\le &\sum_{k\ge l-10}C\|\Om\|_{L^\nf}2^{-2kn}+\sum_{k\le l-10}
C \|\Om\|_{L^\nf}2^{k-N(\ell)} 2^{-(2n+1)l}\\
\le &C\|\Om\|_{L^\nf} [2^{-2ln}+2^{-(2n+1)l}2^{-N(\ell)}2^{l}]\\
\le & C\|\Om\|_{L^\nf}|\vec x|^{-2n}.
\end{align*}

Similarly we can prove that
$$
|\nabla(\sum_{k}K_k*\vp_{k-N(\ell)}(\vec x))|
\le C\|\Om\|_{L^\nf}\tf{2^{N(\ell)}}{|\vec x|^{2n+1}}.
$$

Notice 
that the kernel $\tilde K_\ell(x,y,z)$ of $\tilde T_\ell$ is 
$$\sum_kK_k*[\vp_{k-N(\ell)}-\vp_{k-N(\ell-1)}]((x,x)-(y,z)),$$
so a routine argument implies
\eqref{e07051}
and \eqref{e07052}.
This completes the proof of Lemma~\ref{07051}.




\begin{thebibliography}{99}


\bibitem{Barron2017}
A.~Barron,
{\em Weighted estimates for rough bilinear singular integrals via sparse
  domination},
 arXiv preprint arXiv:1702.04790.


\bibitem{Bergh2012}
J.~Bergh and J.~Lofstrom, 
{\em Interpolation spaces: an introduction},
  Vol.~223, Springer Science \& Business Media, 2012.
  
\bibitem{Coifman1975}
R. R. Coifman, and Y. Meyer,  
\emph{On commutators of singular integrals and bilinear singular integrals}, 
 Trans. Amer. Math. Soc. {\bf 212} (1975), 315--331.

\bibitem{Conde2016}
J. M. Conde-Alonso, A. Culiuc, F. Di Plinio, and Y. Ou, 
{\em A sparse domination principle for rough singular integrals}, 
Anal. PDE { 10} (2017), no. 5, 1255--1284.


\bibitem{Cruz-Uribe2017}
D.~Cruz-Uribe and K.~Moen,
{\em A multilinear reverse {H}\"older
  inequality with applications to multilinear weighted norm inequalities},
  arXiv preprint arXiv:1701.07800,  (2017).

\bibitem{Cruz-Uribe2016}
D.~Cruz-Uribe and V.~Naibo, 
{\em Kato-Ponce inequalities on weighted and variable Lebesgue spaces},
Differential and Integral Equations (9/10) {\bf 29} (2016):801--836.

\bibitem{Culiuc2016}
A. Culiuc, F. Di Plinio, and Y. Ou, 
{\em Domination of multilinear singular integrals by positive sparse
forms}, arXiv preprint arXiv: 1603.05317.

\bibitem{DRdF}
J.~Duoandikoetxea and J.~L.~Rubio de Francia,
{\em Maximal and singular integral operators via Fourier transform estimates}, Invent. Math. {\bf 84} (1986), no. 3, 541--561.


\bibitem{Grafakos2014}
L.~Grafakos, {\em Classical Fourier analysis}, vol.~249, Springer, 2014.

\bibitem{Grafakos2014a}
L.~Grafakos, {\em Modern Fourier analysis}, vol.~250, Springer, 2014.

\bibitem{Grafakos2015}
 L.~Grafakos, D.~He, and P.~Honz{\'\i}k, {\em Rough bilinear singular
  integrals}, arXiv preprint arXiv:1509.06099.
 
\bibitem{Hytonen2012}
T.~Hyt\"onen, 
{\em The sharp weighted bound for general Calder\'on-Zygmund operators}, Ann. of Math. (2) {\bf175} (2012), no. 3, 1473--1506.
 
 \bibitem{Hytonen2013a}
T.~Hyt\"onen and C.~P\'erez, 
{\em Sharp weighted bounds involving $A_\nf$}, Anal. PDE 6 (2013), no. 4, 777--818.
 
 
\bibitem{Hytonen2012a}
T.~Hyt{\"o}nen, C.~ P{\'e}rez, and E.~Rela,
{\em Sharp reverse H\"older property for $A_\nf$ weights on spaces of homogeneous type}, 
J. Funct. Anal. {\bf 263} (2012), no. 12, 3883--3899. 

\bibitem{Hytonen2015}
 T.~ Hyt{\"o}nen, L.~Roncal, and O.~Tapiola, {\em Quantitative weighted
  estimates for rough homogeneous singular integrals}, arXiv preprint
  arXiv:1510.05789.

\bibitem{Krause2016}
B.~Krause and M.~Lacey, 
{\em Sparse Bounds for Random Discrete Carleson Theorems}, 
arXiv preprint
  arXiv: 1609.08701.

\bibitem{Lacey2017}
M.~Lacey,
{\em An elementary proof of the $A_2$ bound}, Israel J. Math. {\bf217} (2017), no. 1, 181--195.

\bibitem{Lacey2016}
M.~Lacey and S.~Spencer, 
{\em Sparse Bounds for Oscillatory and Random Singular Integrals},
arXiv preprint
  arXiv: 1609.06364.

\bibitem{Lerner2013}
A.~Lerner, 
{\em A simple proof of the $A_2$ conjecture}, 
Int. Math. Res. Not. IMRN 2013, no. 14, 3159--3170.

\bibitem{Lerner2016}
A.~Lerner,
{\em On pointwise estimates involving sparse operators}, New York J. Math. 22 (2016), 341--349.

\bibitem{Lerner2009}
A.~Lerner, S.~Ombrosi, C.~P{\'e}rez, R.~H. Torres, and
  R.~Trujillo-Gonz{\'a}lez, 
 {\em New maximal functions and multiple weights
  for the multilinear Calder{\'o}n--zygmund theory}, Advances in Mathematics,
 {\bf 220} (2009), 1222--1264.

\bibitem{Li2016}
K.~Li,
{\em Sparse domination theorem for multiliner singular integral
  operators with $ L^r$-H\"ormander condition}, arXiv preprint
  arXiv:1606.03925.
  
  


\end{thebibliography}
\end{document}